\documentclass[12pt]{amsart}
\input{commands.tex}
\usepackage{microtype}
\usepackage{enumitem}
\usepackage{todonotes}
\usetikzlibrary{patterns}

\providecommand{\semi}{\text{semi}}
\providecommand{\twist}{\text{twist}}

\newcommand{\twR}[1]{R_{\theta}\left[#1\right]}

\usepackage{xcolor}
\definecolor{darkgreen}{rgb}{0,0.30,0} 
\definecolor{darkred}{rgb}{0.75,0,0}
\definecolor{darkblue}{rgb}{0,0,0.6} 
\definecolor{lightblue}{RGB}{179, 230, 255}
\usepackage{amsthm,thmtools}
\usepackage[pdfborder=0,pagebackref,colorlinks,citecolor=darkgreen,linkcolor=darkgreen,urlcolor=darkblue]{hyperref}
\usepackage{cleveref}

\def\makeautorefname#1#2{\expandafter\def\csname#1autorefname\endcsname{#2}}

\makeautorefname{eqn}{Equation}%
\makeautorefname{sec}{Section}%
\makeautorefname{subsec}{Subsection}%
\makeautorefname{footnote}{footnote}%
\makeautorefname{item}{item}%
\makeautorefname{figure}{Figure}%
\makeautorefname{table}{Table}%
\makeautorefname{part}{Part}%
\makeautorefname{app}{Appendix}%
\makeautorefname{cla}{claim}%
\makeautorefname{conj}{conjecture}%
\makeautorefname{cor}{corollary}%
\makeautorefname{cex}{counterexample}%
\makeautorefname{cexs}{counterexamples}%
\makeautorefname{dig}{digression}%
\makeautorefname{disc}{discussion}%
\makeautorefname{def}{definition}%
\makeautorefname{ex}{example}%
\makeautorefname{exs}{examples}%
\makeautorefname{fac}{fact}%
\makeautorefname{intu}{intuition}%
\makeautorefname{lem}{lemma}%
\makeautorefname{nota}{notation}%
\makeautorefname{note}{note}%
\makeautorefname{prop}{proposition}%
\makeautorefname{rmk}{remark}%
\makeautorefname{term}{terminology}%
\makeautorefname{thm}{theorem}%
\makeautorefname{upsh}{upshot}%
\theoremstyle{definition}
\newtheorem{theorem}{Theorem}[section]

\newtheorem{corollary}[theorem]{Corollary}

\newtheorem{definition}[theorem]{Definition}

\newtheorem{proposition}[theorem]{Proposition}
\newtheorem{remark}[theorem]{Remark}

\newtheorem{relatedlit}[theorem]{A remark on related literature}

\makeatletter
\let\c@corollary=\c@theorem
\let\c@proposition=\c@theorem
\let\c@lemma=\c@theorem
\let\c@conjecture=\c@theorem
\let\c@definition=\c@theorem
\let\c@example=\c@theorem
\let\c@remark=\c@theorem
\let\c@notation=\c@theorem
\let\c@equation\c@theorem
\makeatother

\title{A note on twisted group rings and semilinearization}

\author{Thomas Brazelton}

\providecommand{\theauthor}{Thomas Brazelton}

\usepackage{fancyhdr}
\begin{document}
\begin{abstract}
In this short note, we construct a right adjoint to the functor which associates to a ring $R$ equipped with a group action its \textit{twisted group ring}. This right adjoint admits an interpretation as \textit{semilinearization}, in that it sends an $R$-module to the group of semilinear $R$-module automorphisms of the module. As an immediate corollary, we provide a novel proof of the classical observation that modules over a twisted group ring are modules over the base ring together with a semilinear action.
\end{abstract}

\maketitle

\section{Introduction}

Twisted group rings, or \textit{skew group rings}, are classical algebraic objects which provide a way to incorporate a group action on a ring into the multiplication on a group ring. These rings are classical enough to evade any attempt to pin down their origin, but they appear as early as the 1970's, where Handelman, Lawrence, and Schelter began to establish their theory \cite{HLS}. The classical question of which properties of rings and groups extend to properties on the group ring has been a fruitful direction of research (see for example \cite[Appendix~2]{Lambek}), and the same can be said for twisted group rings --- early examples include the theses of Chen \cite{Chen-thesis} and Park \cite{Park-thesis}, as well as \cite{FishMont,Ost}.

Twisted group rings are ubiquitous in modern algebra, appearing in fields as diverse as the representation theory of Lie algebras \cite{Mcconnell} to Heisenberg categorification \cite{RossoSavage}. In equivariant homotopy theory, the algebraic $K$-theory of twisted group rings arises when taking fixed points of the equivariant algebraic $K$-theory spectrum \cite{Mona}. Upcoming work of the author computes the homotopy groups of this spectrum as Mackey functors of $K$-groups of twisted group rings \cite{Brazelton}, which served as the impetus for the work here.

This note provides a concise introduction to twisted group rings. In particular, we see that the association of a twisted group ring to a group action is functorial, and admits a right adjoint. Our main result is as follows.

\begin{theorem} (As \autoref{thm:twistification-semilinearization-adjunction}) For any ring $R$, there is an adjunction of slice categories
\begin{align*}
    \Grp/\Aut(R) \rightleftarrows R/\Ring,
\end{align*}
where the left adjoint sends a group homomorphism $G \to \Aut(R)$ to its twisted group ring.
\end{theorem}

We refer to the right adjoint as \textit{semilinearization}; if $R \to \End_\Ab(M)$ determines an $R$-module, the right adjoint sends it to the group of semilinear $R$-module automorphisms of $M$. As an immediate corollary of the above theorem, we recover the classical observation that modules over a twisted group ring are modules over the base ring equipped with a semilinear action (\autoref{cor:mod-over-tw-grp-rg}).

\begin{relatedlit} In the work of Fechete and Fechete, the authors present a right adjoint to the functor $(\id_\Grp /\Aut) \to\Ring$ from the comma category of groups over automorphisms of rings, associating to a group action $\sigma: G\to \Aut(R)$ the group ring $R_\sigma[G]$ \cite{fechete}. Our approach is morally different, as we view the construction of a twisted group ring over $R$ as naturally valued in $R$-algebras. In particular our right adjoint is quite different than the one constructed by Fechete and Fechete, and we interpret it in the context of semilinear module automorphisms. As another point, we would like to draw the attention of the reader to excellent exposition found in the unpublished thesis of Edward Poon on categorical aspects of twisted group rings \cite{poon}.
\end{relatedlit}

\section{Semilinear $G$-actions}

\begin{definition} Let $R$ be a ring, and $\phi \in \Aut_{\Ring}(R)$ a ring automorphism of $R$. We define a $\phi$-\textit{semilinear} $R$-module homomorphism $f: M \to N$ to be a function satisfying
\begin{enumerate}
    \item $f(m+m') = f(m) + f(m')$ for all $m,m' \in M$,
    \item $f(rm) = \phi(r)\cdot f(m)$ for all $r\in R$ and $m\in M$.
\end{enumerate}
\end{definition}
In particular, semilinear isomorphisms are precisely those maps which are underlain by bijections.

The collection of $R$-modules with $\phi$-semilinear maps \textit{does not} form a category unless $\phi$ is trivial (we may see that there is no good notion of identity morphism). We may, however, create a category by considering \textit{all} semilinear maps, that is, all morphisms which are $\phi$-semilinear for some choice of $\phi$. We see that identity morphisms are $\id_R$-semilinear, and that morphisms compose; given $f: M \to N$ which is $\phi$-semilinear and $g: N \to P$ which is $\omega$-semilinear, then $gf$ is $(\omega\circ \phi)$-semilinear. We let $\Mod^\semi(R)$ denote this category. Note that $\Mod(R) \subseteq \Mod^\semi(R)$ is the subcategory of $\id_R$-semilinear morphisms.

\begin{remark} There is always a forgetful functor
\begin{align*}
    \Mod^\semi(R) \to \Aut_\Ring(R),
\end{align*}
given by sending a $\phi$-semilinear morphism $f$ to $\phi$. In particular, this induces a group homomorphism $U: \Aut_{\Mod^\semi(R)}(M) \to \Aut_\Ring(R)$ for any $M\in \Mod(R)$.
\end{remark}

\begin{definition} Let $R$ be a ring. A \textit{semilinear} $G$\textit{-action} on an $R$-module $M$ is a group homomorphism $G \to \Aut_{\Mod^\semi(R)}(M)$.
\end{definition}
Such a semilinear action determines a unique $G$-ring structure on $R$ by post-composition with the natural homomorphism $U:\Aut_{\Mod^\semi(R)}(M) \to \Aut_\Ring(R)$. In practice, we will care about the case where $R$ is already equipped with a $G$-ring structure via some group homomorphism $\theta: G \to \Aut_\Ring(R)$. In this setting, the definition of a semilinear $G$-action is given by an appropriate choice of lift of $\theta$ along $U$.

\begin{definition} Let $R$ be a $G$-ring. A \textit{semilinear} $G$\textit{-action} on an $R$-module $M$ is a group homomorphism $f: G \to \Aut_{\Mod^\semi(R)}(M)$ making the diagram commute
\[ \begin{tikzcd}[row sep=large]
     & G\ar[dl,dashed,"f" above left]\dar["\theta" right]\\
     \Aut_{\Mod^\semi(R)}(M)\rar["U" below] & \Aut_\Ring(R),
\end{tikzcd} \]
where $\theta$ is the $G$-action on $R$. Explicitly, this is the data of a $\theta_g$-semilinear module homomorphism $f(g)$ for every $g\in G$ so that $f(e)$ is the identity and $f(gh) = f(g)f(h)$ is $\theta_{gh}$-semilinear.
\end{definition}

\section{Twisted group rings and semilinearization}\label{sec:semilinearization}

Given an $R$-module $M$ with multiplication $\chi: R \to \End_\Ab(M)$, we can functorially determine its group of semilinear $R$-module automorphisms. We define $\semi_R(\End_\Ab(M))$ to be the following group:
\begin{align*}
        \semi_R(\End_\Ab(M)) &:= \left\{ (\sigma,\phi) \in \End_\Ab(M)^\times \times \Aut_{\Ring}(R) \mid \sigma \circ\chi(r) = \chi(\phi(r))\cdot \sigma\ \forall r\in R \right\} \\
        &= \bigcup_{\phi\in \Aut_{\Ring}(R)} \left\{ \phi\text{-semilinear } R\text{-module automorphisms of } M\right\} \\
        &= \Aut_{\Mod^\semi(R)}(M).
\end{align*}

Briefly forgetting that $\End_\Ab(M)$ is an endomorphism ring of an abelian group, we can replicate the construction above for an arbitrary ring $S$, provided that $S$ comes with of a ring homomorphism from $R$. This motivates the following definition.

\begin{definition}\label{def:group-semilinear-automorphisms} Let $R$ be a ring, and suppose that $f: R \to S$ is an object of the slice category $R/\Ring$. Then we define the \textit{group of semilinear inner automorphisms under} $R$ as
\begin{align*}
    \semi_R(S) := \left\{ (s,\phi) \in S^\times \times \Aut_\Ring(R) \mid sf(r) = f(\phi(r))s \text{ for each } r\in R \right\}.
\end{align*}
One may easily verify that this is a group, where multiplication occurs diagonally as $(s,\phi)\cdot (s',\psi) := (ss',\phi\psi)$. We remark that $\semi_R(S)$ has a forgetful group homomorphism to $\Aut_\Ring(R)$.
\end{definition}

\begin{proposition}\label{prop:semilinearization-is-functor} For a ring $R$, there is a functor of slice categories
\begin{align*}
    \semi_R: R/\Ring \to \Grp/\Aut_\Ring(R),
\end{align*}
which we call \textit{semilinearization}.
\end{proposition}
\begin{proof} Suppose we have a morphism in $R/\Ring$ of the form
\[ \begin{tikzcd}
     & R\ar[dl,"f" above left]\ar[dr,"g" above right] & \\
    S\ar[rr,"h" below] &  & T.
\end{tikzcd} \]
Then we define
\begin{align*}
    \semi_R(h): \semi_R(S) &\to \semi_R(T) \\
    (s,\phi) &\mapsto (h(s),\phi).
\end{align*}
To verify that $(h(s),\phi) \in \semi_R(T)$, we see that
\begin{align*}
    h(s) g(r) &= h(s) h(f(r)) = h(sf(r)) = h(f\phi(r)s) = g(\phi(r))h(s),
\end{align*}
for any $r\in R$. It is clear that $\semi_R(h)$ is a group homomorphism, and moreover that it commutes with the forgetful maps from $\semi_R(S)$ and $\semi_R(T)$ to $\Aut_\Ring(R)$. It is straightforward to check that $\semi_R(-)$ preserves identities, composition, and associativity, and thus defines a functor. 
\end{proof}

\begin{definition}\label{def:twisted-group-ring-functor} Let $R$ be a ring, and let $\theta : H \to \Aut_\Ring(R)$ be an element of the slice category $\Grp/\Aut_\Ring(R)$. The \textit{twisted group ring} $R_\theta[H]$ has the same elements as the group ring $R[H]$, but the multiplication is twisted by $\theta$ in the following way:
\begin{align*}
    (r_1 h_1) \cdot (r_2 h_2) := r_1 \theta_{h_1}(r_2) h_1 h_2,
\end{align*}
where we understand that this definition extends additively. The ring $R_\theta[H]$ comes equipped with a natural ring homomorphism $R \to R_\theta[H]$ sending $r$ to $r1_H$.
\end{definition}

\begin{proposition}\label{prop:twistification-is-a-functor} The assignment of a twisted group ring to a group action assembles into a functor
\begin{align*}
    \twist_R : \Grp/\Aut_\Ring(R) \to R/\Ring,
\end{align*}
which we refer to as \textit{twistification}.
\end{proposition}
\begin{proof} To define $\twist_R$ on morphisms, suppose we have a morphism in $\Grp/\Aut_\Ring(R)$:
\[ \begin{tikzcd}
    G\ar[rr,"f" above]\ar[dr,"\theta" below left] &  & K\ar[dl,"\psi" below right]\\
     & \Aut_\Ring(R), &
\end{tikzcd} \]
meaning that $\theta_g(r) = \psi_{f(g)}(r)$ for any $r\in R$. Then there is a function $F: R_\theta[G] \to R_\psi[H]$ sending $rg \mapsto rf(g)$. We verify that $F$ is a ring homomorphism by observing that
\begin{align*}
        F((r_1 g_1)\cdot (r_2 g_2)) &= F(r_1 \theta_{g_1}(r_2) g_1 g_2) = r_1 \psi_{f(g_1)}(r_2) f(g_1) f(g_2) \\
        &= \left( r_1 f(g_1) \right)\cdot \left( r_2 f(g_2) \right) = F(r_1 g_1)\cdot F(r_2 g_2).
\end{align*}
It is immediate to check this assignment is functorial.
\end{proof}

\begin{theorem}\label{thm:twistification-semilinearization-adjunction} Twistification and semilinearization define an adjunction
\begin{align*}
    \twist_R : \Grp/\Aut_\Ring(R) \leftrightarrows R/\Ring : \semi_R.
\end{align*}

\end{theorem}

The final section of this note is dedicated to the proof of \autoref{thm:twistification-semilinearization-adjunction}. We will include one corollary of the natural bijection associated to this adjunction.

\begin{corollary}\label{cor:mod-over-tw-grp-rg} \textit{(Modules over a twisted group ring are $R$-modules with semilinear $G$-action)} Let $R$ be a $G$-ring. Then for any $R$-module $\End_\Ab(M) \in R/\Ring$, there is a natural isomorphism
\begin{align*}
    \Hom_{R/\Ring}\left(\twR{G}, \End_\Ab(M) \right) \cong \Hom_{\Grp/\Aut_\Ring(R)} \left(G, \Aut_{\Mod^\semi(R)}(M) \right).
\end{align*}
The left side describes extensions of the $R$-module structure on $M$ to an $\twR{G}$-module structure, while the right side describes the possible semilinear $G$-actions on $M$ compatible with the $G$-action on $R$.
\end{corollary}

\section{Proof of~\autoref{thm:twistification-semilinearization-adjunction}}
In order to verify the natural bijection for the adjunction in \autoref{thm:twistification-semilinearization-adjunction}, we will treat the bijection and naturality separately for ease of reading.

\begin{proposition} For any ring homomorphism $\chi: R \to S$ and group homomorphism $\theta: G \to\Aut_\Ring(R)$ there is a bijection
\begin{align*}
    \Pi: \Hom_{R/\Ring}(\twR{G}, S) &\xto{\sim} \Hom_{\Grp/\Aut_\Ring(R)}(G, \semi_R(S)),
\end{align*}
given by sending a ring homomorphism $f: R_\theta[G] \to S$ to the group homomorphism $g \mapsto \left( f(1_Rg), \theta_g \right)$.
\end{proposition}
\begin{proof} We first check that $\left( f(1_R g), \theta_g \right) \in \semi_R(S)$ for any $g\in G$. That is, we must see that $f(1_R g)\chi(r)$ is equal to $\chi(\theta_g(r))f(1_R g)$ for any $r\in R$. Recall that $f$ is a homomorphism under $R$, meaning that $f(r e_G) = \chi(r)$ for any $r\in R$. Thus
\begin{align*}
    f(1_R g) \chi(r) &= f(1_R g) f(r e_G) = f \left( (1_R g)\cdot (re_G) \right) = f(1_R \theta_g(r) g e_G) \\
    &= f \left( \theta_g(r) g \right) = f \left( (\theta_g(r) e_G)\cdot(1_R g) \right) \\
    &= \chi(\theta_g(r)) f(1_R g).
\end{align*}

We observe that the function $g \mapsto \left( f(1_Rg), \theta_g \right)$ is a group homomorphism since $\theta_{gg'} = \theta_g \theta_{g'}$, and $f(1_R g g') = f(1_R g)f(1_R g')$. Finally, we remark that $\Pi(f)$ is a group homomorphism over $\Aut_\Ring(R)$, so $\Pi$ is well-defined.

Next we check that $\Pi$ is injective. Suppose $f$ and $h$ are two ring homomorphisms $\twR{G} \to S$ under $R$ with the property that $\Pi(f) = \Pi(h)$. By definition, we see that $f(re_G) = \chi(r) = h(re_G)$ for each $r\in R$. For any $r\in R$ and $g\in G$, we have that $(re_G) \cdot (1_R g) = r\theta_{e_G}(1_R) e_G g = rg$. Thus $f(rg) = f(re_G)f(1_R g)$, so it suffices for us to check that $f(1_R g) = h(1_R g)$ for each $g\in G$. This is clearly true when $\Pi(f) = \Pi(h)$. So $f$ and $h$ agree, and hence $\Pi$ is injective.

Finally we verify that $\Pi$ is surjective. As $\semi_R(S)$ is by definition a subset of the product $S^\times \times \Aut_\Ring(R)$, it comes equipped with natural projection maps
\begin{align*}
    \begin{tikzcd}[ampersand replacement=\&]
    S^\times \& \semi_R(S)\ar[l,"\pi" above]\ar[r,"\pi_s" above] \& \Aut_\Ring(R),\end{tikzcd}
\end{align*}
which are group homomorphisms as $\semi_R(S)$ is endowed with diagonal multiplication. Suppose that $\alpha : G \to \semi_R (S)$ is any group homomorphism over $\Aut_\Ring(R)$. Then $\pi_s(\alpha(g)) = \theta_g$, so $\alpha$ is really determined by the data of $\pi \alpha: G \to S^\times$. We observe that
\begin{equation}\label{eqn:alpha-in-semi}
\begin{aligned}
    (\pi\alpha(g))\chi(r) = \chi(\theta_g(r)) (\pi\alpha(g)) \quad\quad \text{for any }r\in R,
\end{aligned}
\end{equation}
since $\alpha(g) = \left( \pi\alpha(g), \theta_g \right)$ is an element of $\semi_R(S)$.

Our goal is to concoct a ring homomorphism $f: \twR{G} \to S$ in $R/\Ring$ with the property that $\Pi(f) = \alpha$. Such an $f$ must satisfy
\begin{enumerate}
    \item $f(re_G) = \chi(r)$ since $f$ is a morphism under $R$
    \item $f(1_R g) = \pi\alpha(g)$ since $\Pi(f) = \alpha$.
\end{enumerate}
As we notice that $f(rg) = f((re_G)\cdot(1_R g))$ for any $r$ and $g$, it makes sense to define $f$ to be the function
\begin{align*}
    f: \twR{G} &\to S \\
    rg &\mapsto \chi(r) \pi\alpha(g).
\end{align*}
If we can verify that $f$ is a ring homomorphism under $R$, then $\Pi(f) = \alpha$, and we will have verified that $\Pi$ is surjective. Clearly $f$ preserves multiplicative and additive identities, and preserves addition by definition, so it suffices to check it is multiplicative. Let $r_1, r_2 \in R$, and $g_1, g_2 \in G$. Then
\begin{align*}
    f \left( (r_1 g_1)(r_2 g_2) \right) &= f \left( r_1 \theta_{g_1}(r_2) g_1 g_2 \right) = \chi(r_1 \theta_{g_1}(r_2)) \pi\alpha(g_1 g_2) \\
    &= \chi(r_1) \chi(\theta_{g_1}(r_2)) \pi\alpha(g_1) \pi\alpha(g_2).
\end{align*}
By \autoref{eqn:alpha-in-semi}, we see that the above is equal to
\begin{align*}
    &\chi(r_1) \left(\chi(\theta_{g_1}(r_2)) \pi\alpha(g_1)\right) \pi\alpha(g_2) = \left(\chi(r_1) \pi\alpha(g_1)\right) \left( \chi(r_2) \pi\alpha(g_2) \right) = f(r_1 g_1) f(r_2 g_2).\qedhere
\end{align*}
\end{proof}

\begin{proposition}\label{prop:natural-bijection} The function $\Pi$ is a natural bijection.
\end{proposition}
\begin{proof} Let $j : G \to K$ be a group homomorphism over $\Aut(R)$ and let $h : S \to T$ be a ring homomorphism under $R$. For naturality, it suffices to verify that the left square commutes if and only if the right square commutes:
\begin{equation}\label{eqn:two-squares}
\begin{aligned}
    \begin{tikzcd}[baseline=0cm]
     {R_\theta[G]}\dar["\twist_R(j)" left]\rar["\mu"] & S\dar["h" right]\\
    {R_\psi[K]}\rar["\lambda" below] & T
\end{tikzcd} \ \Longleftrightarrow\  \begin{tikzcd}[baseline=0cm]
    G\dar["j" left]\rar["\Pi(\mu)"] & \semi_R(S)\dar["\semi_R(h)"]\\
    K\rar["\Pi(\lambda)" below] & \semi_R(T).
\end{tikzcd}
\end{aligned}
\end{equation}
As we are asking for commutativity in a slice category, we know that the left hand square already commutes on certain objects of $R_\theta[G]$, namely those arriving the structure map from $R$, which are elements of the form $re_G$. Thus
\begin{align*}
    h(\mu(re_G)) = \lambda(\twist_R(j)(r e_G)) = \lambda(r j(e_G)) = \lambda(re_K),
\end{align*}
for all $r \in R$. Since this is a diagram of ring homomorphisms, the question of commutativity reduces to a question of commutativity on the elements in $\twR{G}$ of the form $1_R g$ for $g\in G$.

A directly analogous statement is true about the right square in \autoref{eqn:two-squares}. As this is a diagram over $\Aut_\Ring(R)$, we do not need to worry about compatibility with projection maps $\pi_s: \semi_R(S) \to \Aut_\Ring(R)$. Thus we can restrict our attention to $\pi \semi_R(S) \subseteq S^\times$. Phrased differently, it suffices to verify that the following diagram commutes
\[ \begin{tikzcd}
    G\rar["\pi\Pi(\mu)"]\dar["j" left] & S^\times \dar["\left. h\right|_{S^\times}" right]\\
    K\rar["\pi\Pi(\lambda)" below] & T^\times.
\end{tikzcd} \]
As $\Pi(\mu)(g) = \left( \mu(1_Rg), \theta_g \right)$, we see that $\pi\Pi(\mu) = \mu(1_R -)$, and similarly $\pi \Pi(\lambda) = \lambda(1_R -)$. Therefore the above diagram can be rewritten as
\[ \begin{tikzcd}
    G\rar["\mu(1_R-)" above]\dar["j" left] & S^\times \dar["h" right]\\
    K\rar["\lambda(1_R -)" below] & T^\times.
\end{tikzcd} \]
The commutativity of this diagram is equivalent to the condition that $h\mu$ and $\lambda j$ agree on elements of the form $1_R g$, which as we have already seen is equivalent to the commutativity of the left square in \autoref{eqn:two-squares}.\end{proof}

\section*{Acknowledgements}
The author would like to thank Mona Merling for guidance throughout this note, and Maxine Calle for helpful edits. The author is supported by an NSF Graduate Research Fellowship (DGE-1845298).

\bibliographystyle{alpha}
\bibliography{semilin.bib}{}

\begin{thebibliography}{McC75}

\bibitem[Bra21]{Brazelton}
Thomas Brazelton.
\newblock Homotopy groups of equivariant algebraic {K}-theory.
\newblock {\em In progress}, 2021.

\bibitem[Che78]{Chen-thesis}
Nan-Hung Chen.
\newblock {\em Global dimension of skew group rings}.
\newblock ProQuest LLC, Ann Arbor, MI, 1978.
\newblock Thesis (Ph.D.)--Rutgers The State University of New Jersey - New
  Brunswick.

\bibitem[FF09]{fechete}
I.~Fechete and D.~Fechete.
\newblock Some categorial aspects of the skew group rings.
\newblock {\em An. Univ. Oradea Fasc. Mat.}, 16:197--207, 2009.

\bibitem[FM78]{FishMont}
Joe~W. Fisher and Susan Montgomery.
\newblock Semiprime skew group rings.
\newblock {\em J. Algebra}, 52(1):241--247, 1978.

\bibitem[HLS78]{HLS}
David Handelman, John Lawrence, and William Schelter.
\newblock Skew group rings.
\newblock {\em Houston J. Math.}, 4(2):175--198, 1978.

\bibitem[Lam66]{Lambek}
Joachim Lambek.
\newblock {\em Lectures on rings and modules}.
\newblock With an appendix by Ian G. Connell. Blaisdell Publishing Co. Ginn and
  Co., Waltham, Mass.-Toronto, Ont.-London, 1966.

\bibitem[McC75]{Mcconnell}
J.~C. McConnell.
\newblock Representations of solvable {L}ie algebras. {II}. {T}wisted group
  rings.
\newblock {\em Ann. Sci. \'{E}cole Norm. Sup. (4)}, 8(2):157--178, 1975.

\bibitem[Mer17]{Mona}
Mona Merling.
\newblock Equivariant algebraic {K}-theory of {$G$}-rings.
\newblock {\em Math. Z.}, 285(3-4):1205--1248, 2017.

\bibitem[Ost79]{Ost}
James Osterburg.
\newblock The coefficient ring of the skew group ring.
\newblock {\em Czechoslovak Math. J.}, 29(104)(1):144--147, 1979.

\bibitem[Par78]{Park-thesis}
Jae~Keol Park.
\newblock {\em Artinian skew group rings and semiprime twisted group rings}.
\newblock ProQuest LLC, Ann Arbor, MI, 1978.
\newblock Thesis (Ph.D.)--University of Cincinnati.

\bibitem[Poo16]{poon}
Edward Poon.
\newblock {\em SKEW GROUP RINGS}, 2016.
\newblock \url{https://alistairsavage.ca/pubs/Poon-Skew_group_rings.pdf}.

\bibitem[RS17]{RossoSavage}
Daniele Rosso and Alistair Savage.
\newblock A general approach to {H}eisenberg categorification via wreath
  product algebras.
\newblock {\em Math. Z.}, 286(1-2):603--655, 2017.

\end{thebibliography}

\end{document}